\documentclass[11pt, english, a4paper]{amsart}
\usepackage[margin=3cm]{geometry}
\usepackage{amssymb,amsmath,amsthm}
\usepackage{mathtools}
\usepackage{mathrsfs}
\usepackage{accents}
\usepackage[utf8]{inputenc}
\usepackage[T1]{fontenc}

\usepackage{caption}
\usepackage{subcaption}

\usepackage{xfrac}
\usepackage{enumerate}
\usepackage{graphics, float}

\usepackage[colorlinks]{hyperref}
\hypersetup{allcolors = black}

\usepackage{url}
\usepackage[T1]{fontenc} 

\usepackage{svg}

\usepackage{comment}

\newtheorem{theorem}{Theorem}
\newtheorem{lemma}{Lemma}
\newtheorem{proposition}{Proposition}
\newtheorem{corollary}{Corollary}

\newcommand{\C}{\mathbb{C}}
\newcommand{\R}{\mathbb{R}}
\newcommand{\la}{\langle}
\newcommand{\ra}{\rangle}
\newcommand{\I}{\mathcal{I}}

\newcommand{\Diag}{\mathrm{Diag}\,}
\newcommand{\D}{\mathrm{D}}
\newcommand{\E}{\mathcal{E}}
\newcommand{\eps}{\varepsilon}

\newcommand{\tf}{\widetilde{f}}

\newcommand{\HS}{\mathrm{HS}}

\newcommand{\dd}{\,\mathrm{d}}

\newcommand{\res}{\operatorname{res}}
\newcommand{\Res}{\operatorname{Res}}
\newcommand{\tr}{\operatorname{tr}}
\newcommand{\id}{\mathrm{id}}

\newcommand{\LL}{\mathcal{L}}

\title{Isotropic Decompositions via Inverse Eigenvectors}
\author{Gergely Ambrus}
\date{\today}

\thanks{Research of was partially supported by the ERC Advanced Grant "GeoScape" no.  882971, by Hungarian National Research (NKFIH) grants no. KKP-133819, 147145, 147544, and 150151, which has been implemented with the support provided by the Ministry of Culture and Innovation of Hungary from the National Research, Development and Innovation Fund, financed under the ADVANCED-24 funding scheme. This research was funded by the grant 2024-1.2.8-TÉT-IPARI-CN-2025-00011,
with the support provided by the National Research,
Development and Innovation Office from the National Research,
Development and Innovation Fund, and financed under the
2024-1.2.8-TÉT-IPARI-CN funding scheme. }

\begin{document}

\begin{abstract}
We develop a residue-theoretic framework for studying inverse eigenvectors of a square matrix, defined by the nonlinear equation $M\alpha=\alpha^{-1}$. Our main result is an inverse analogue of the spectral theorem: under natural transversality and properness assumptions, the identity operator admits an explicit decomposition into rank-one tensors associated with the inverse eigenvectors. The proof is based on residues of rational differential forms in several complex variables. For real correlation matrices whose off-diagonal entries have modulus strictly less than one, we remove the properness assumption and show that the coefficients in the decomposition are positive and sum to~$1$. Consequently, the inverse eigenvectors support an explicit centered discrete isotropic probability measure and form a weighted tight frame. We extend the construction to arbitrary real Gram matrices, diagonal inverse eigenvectors, and weighted inverse-eigenvector equations. Simple trace and arithmetic--geometric mean arguments yield inverse eigenvectors with controlled Euclidean norm and coordinate product. These estimates lead to short proofs of the strong real polarization inequality and the $n$th real linear polarization inequality, together with weighted and matrix-valued generalizations and further geometric and analytic applications.
\end{abstract}

\maketitle

\section{Introduction}\label{section:intro}

In a recent breakthrough article \cite{MartinezOrtegaMoreno2026}, Martínez and Ortega-Moreno resolved the polarization problem, which had remained open for several decades. A central tool in their proof is the Euler--Jacobi vanishing theorem, an algebraic-geometric consequence of the Global Residue Theorem for meromorphic forms in several complex variables.

The aim of the present article is to combine this residue-theoretic approach with the notion of \emph{inverse eigenvectors} of square matrices. This notion first appeared in Ball's solution of the complex plank problem \cite{ball2001complex} and was formally introduced in the author's PhD thesis \cite{Ambrus2009}. Applying residue theory in this setting not only provides a more transparent approach to the polarization problem, but also yields a substantially stronger result reminiscent of John's theorem \cite{john1948extremum}, which characterizes convex bodies whose maximal-volume inscribed ellipsoid is the unit ball through a decomposition of the identity operator.

We prove that, under mild assumptions on the matrix $M \in \C^{n \times n}$, the identity operator admits an explicit decomposition into rank-one tensors associated with the inverse eigenvectors of $M$. The resulting identity may be viewed as an inverse analogue of the spectral decomposition of a positive semidefinite matrix,  and we therefore refer to it as the \emph{Inverse Spectral Theorem}. We also derive several consequences of this theorem, most notably the polarization inequalities.

\medskip

We briefly introduce the relevant terminology. Throughout, $M = (m_{kj})$, where $k,j \in [n]:= \{ 1, \ldots, n\}$, denotes a $n \times n$ real or complex matrix. The spaces of such matrices are denoted by $\R^{n \times n}$ and $\C^{n \times n}$, respectively. To facilitate the use of tools from algebraic geometry, we assume throughout that the ambient dimension satisfies $n\ge2$.

We often suppress the dimension $n$ from the notation; in particular, the ranges of sums and products will be omitted when they are clear from the context. 
The spectrum of $M$ is denoted by $\sigma(M)$.

A real symmetric or complex Hermitian positive semidefinite matrix will be referred to as a \emph{Gram matrix}. Such a matrix arises from a family of vectors $v_1,\ldots,v_n$ in a real or complex inner-product space through $m_{kj} = \la v_k, v_j \ra$.
If, in addition, $m_{kk}=1$ for every $k\in[n]$, then $M$ is called a \emph{correlation matrix}. Strictly positive definite matrices will simply be called \emph{positive}.

For a square matrix $A$, we write $\Diag(A)$ for the diagonal matrix whose diagonal agrees with that of $A$. Conversely, for $x\in\C^n$, $\D(x)$ denotes the diagonal matrix with diagonal $x$.  The identity operator on $\C^n$ is denoted by $\id_n$, with $I_n$ being its matrix in the standard basis. Thus, every correlation matrix $M$ satisfies $\Diag(M) = I_n$.

The notation $|\cdot|$ stands for the absolute value of a real or complex scalar, as well as the Euclidean norm of a real or complex vector. For a finite set $S$, its cardinality is denoted by $\#S$.

\smallskip

For vectors $x, y \in \C^n$, their Hadamard product $x \odot y$ is defined by coordinatewise multiplication: $(x \odot y)_j = x_j y_j$. The Hadamard product of matrices is defined analogously. 
The tensor product $x \otimes y$ is defined as the rank-one linear operator on $\C^n$ represented by the matrix
\[
(x \otimes y) = x y^\top.
\]
We emphasize that this is the bilinear algebraic tensor product and differs from the Hermitian rank-one operator $x y^*$. 

For square matrices $A,B \in \C^{n \times n}$, we define their Hilbert--Schmidt inner product by
\[
\la A,B\ra_{\HS} = \tr(B^*A) = \sum_{k,l} a_{kl}\overline{b_{kl}}.
\]
Here $B^*$ denotes the conjugate transpose of $B$. With this convention,
the inner product is linear in the first variable and conjugate-linear in the
second. The associated Hilbert--Schmidt norm is
\[
\|A\|_{\mathrm{HS}} =\sqrt{\la A,A\ra_{\HS}} = \sqrt{\tr(A^*A)} = \Big( \sum_{k,l} |a_{kl}|^2 \Big)^{1/2}.
\]
We equip $\C^{n \times n}$ with its usual Euclidean topology, which is induced by the Hilbert-Schmidt norm.  
For a square matrix $A \in \C^{n \times n}$, we also define its induced matrix $l_\infty$-norm by
\[
\| A \|_\infty = \max_k \sum_j |a_{kj}|.
\]

By an $n$-variate polynomial, we mean a polynomial in $n$ complex variables, with complex coefficients, viewed as a function on $\C^n$. The degree of a monomial is the sum of the exponents of its variables, and the degree $\deg g$ of a polynomial $g$ is the maximum degree among its monomials with nonzero coefficient. A polynomial is homogeneous if all of its monomials have the same degree. 

For a polynomial $g$, its homogeneous part of degree $k$  is the sum of all monomials of $g$ of degree $k$; it is defined to be~$0$ if $g$ contains no such monomials.
The \emph{leading homogeneous part} of $g$ is its homogeneous part of degree $\deg g$. 

Furthermore, if $f_1, \ldots, f_n: \C^n \to \C$ are polynomials, then $f=(f_1, \ldots, f_n): \C^n \to \C^n$ is called a {\em polynomial mapping}. Its leading homogeneous part is defined as $\tf = (\tf_1, \ldots, \tf_n)$, where $\tf_j$ denotes the leading homogeneous part of $f_j$.

\smallskip

For a mapping $f:\C^n \to \C^m$, we denote its zero set by $Z(f) = f^{-1}(\{0\})$. 
If $f: \C^n \to \C^n$, we write $J_f(x)$ for the Jacobian determinant of $f$ at $x$. In this case, a zero $\alpha \in Z(f)$ is called a simple zero if $J_f(\alpha) \neq 0$.

\smallskip

For notions from algebraic topology, we follow \cite{Hatcher2002}. For background in algebraic geometry, we refer to the classical monograph \cite{griffiths1978principles}, while for residue theory, we use \cite{cattani2005introduction} and~\cite{Tsikh1992}.

\medskip

After these technical preparations,  we introduce the central object of the article, following~\cite{Ambrus2009}. 
For a vector $\alpha=(\alpha_1, \ldots, \alpha_n)\in \C^n$ with nonzero coordinates, we define its {\em inverse} with respect to the Hadamard product, i.e. to be taken coordinatewise:
\[
\alpha^{-1}
=
\left(\frac1{\alpha_1},\dots,\frac1{\alpha_n}\right).
\]
A vector $\alpha \in \C^n$ is called an {\em inverse eigenvector} of $M$ if 
\begin{equation}\label{eq:iev_def}
M \alpha = \alpha^{-1}.
\end{equation}
Equivalently, $\alpha$ satisfies the system of polynomial equations
\begin{equation}\label{eq:iev_pol}
    \alpha_k \sum_j m_{kj} \alpha_j = 1, \ \ k \in [n].
\end{equation} 
The set of inverse eigenvectors of a matrix $M$ is denoted by $\I(M)$.

Introducing the quadratic polynomials
\begin{equation}\label{eq:fkdef}
f_{M,k}(x) = x_k (M x)_k -1
\end{equation}
on $C^n$, for $k\in[n]$, we readily see that $\I(M) = Z(f_{M,1}) \cap \cdots \cap Z(f_{M,n})$. Equivalently, 
\begin{equation}\label{eq:IEV_locus}
\I(M)=Z(f_M)
\end{equation}
is the zero locus of the polynomial mapping $f_M:\C^n\to\C^n $  given by 
\begin{equation}\label{eq:fmdef}
f_M=(f_{M,1},\ldots,f_{M,n})
\end{equation}
Consequently, $\I(M)$ is an algebraic subset of $\C^n$. Moreover, directly from the defining equations, it is centrally symmetric with respect to the origin. 

A variant of inverse eigenvectors appeared in K. Ball's proof of the complex plank problem \cite{ball2001complex}. In the complex setting, however, the definition used there differs substantially from the present one: on the right-hand side of \eqref{eq:iev_def}, $\alpha^{-1}$ is replaced by $\overline{\alpha}^{-1}$. Consequently, the corresponding set of inverse eigenvectors  can no longer be interpreted as the zero set of a polynomial mapping. Inverse eigenvectors have also been studied and applied in \cite{Ambrus2009, LeungLiRakesh2007,leung2008,OrtegaMoreno2021}.

\medskip

There are two important structural properties of the mapping $f_M$ on which the subsequent arguments rely. 

The first is the regularity condition that $0$ be a regular value of $f_M$, meaning that the hypersurfaces $Z(f_{M,1}),\ldots,Z(f_{M,n})$ intersect transversely at every $\alpha \in \I(M)$. Equivalently, every inverse eigenvector is a simple zero of $f_M$: $J_M(\alpha) \neq 0$ for every $\alpha \in \I(M)$. Here and throughout, for ease of notation,  $J_M(x)$ stands for the Jacobian determinant of $f_M$ at $x \in \C^n$. Under this assumption, $\I(M)$ is a zero-dimensional, and hence finite, algebraic set consisting entirely of simple zeros. We then call both the matrix $M$ and the corresponding polynomial mapping $f_M$ \emph{transverse}.

An explicit formula for $J_M(\alpha)$ yields the following characterization of transversality, which will be established in Section~\ref{section:matrices}:
\begin{proposition}\label{prop:transverse}
The matrix $M \in \C^{n \times n}$ is transverse if and only if 
    \[
    \det\bigl(I_n +\D(\alpha) M\D(\alpha) \bigr) \neq 0
    \]
for every $\alpha \in \I(M)$. Equivalently, $-1 \not \in \sigma(D(\alpha) M\D(\alpha))$ for any $\alpha \in \I(M)$.
\end{proposition}

The second property asserts that $|f_M(x)| \to \infty$ as $|x| \to \infty$. Equivalently, $f_M^{-1}(C)$ is compact in $\C^n$ for every compact set $C \subset \C^n$. If this condition holds, we call both the matrix $M$ and the corresponding polynomial mapping $f_M$  {\em proper}. In Section~\ref{section:matrices}, we prove the following characterization.

\begin{proposition}\label{prop:Mproper}
The  matrix $M \in \C^{n \times n}$ is  proper if and only if every principal minor of $M$ is nonzero.
\end{proposition}

In particular, positive matrices are proper. 

There is a notable difference between the two characterizations. Properness may be viewed as a strong form of nonsingularity and can be  checked directly by computing the principal minors. Transversality, by contrast, is more difficult to verify, since its characterization requires determining the inverse eigenvectors of $M$.  The next three statements, proved in Section~\ref{section:matrices}, establish the required properties for two important classes of matrices.

\begin{lemma}\label{lemma:realgram_realIEV}
Let $M \in \R^{n \times n}$ be a real Gram matrix. Then  $\I(M) \subset \R^n$.
\end{lemma}

\begin{lemma}\label{lemma:realgram_transverse}
Every real Gram matrix $M \in \R^{n \times n}$ is transverse.  
\end{lemma}

\begin{lemma}\label{lemma:dominant_transverse}
Let $M \in \C^{n \times n}$ satisfy $\Diag M = I_n$ and $\| M - I_n \|_\infty < \frac 1 2$. Then $M$ is transverse and proper.
\end{lemma}
\noindent 
Since $\Diag M = I_n$, the condition $\| M - I_n \|_\infty < \frac 1 2$ is equivalent to requiring that for every $k \in [n]$,
\[
\sum_{j \neq k} |m_{k j}| < \frac 1 2.
\]
In particular, $M$ is strictly diagonally dominant.  
\medskip

If $M$ is transverse, then $\I(M)=Z(f_M)$ is a finite set of simple zeros. Since each component of $f_M$ has degree two, Bézout's theorem yields
\begin{equation*}\label{eq:IEV_cardinality}
\# \I(M) \leq 2^n.
\end{equation*}
When $M$ is a real Gram matrix, the set $\I(M)$ admits a simple geometric description obtained by variational methods. Let $p$ denote the homogeneous polynomial of degree $n$ defined by
\begin{equation*}\label{eq:pdef*}
p(x)=\prod_{j} x_j
\end{equation*}
for $x = (x_1, \ldots, x_n) \in \R^n$. 
The following characterization was proved for nonsingular matrices in \cite[Lemma 4]{LeungLiRakesh2007} and in  full generality in \cite[Proposition 3]{leung2008}.

\begin{proposition}[\cite{Ambrus2009,LeungLiRakesh2007,leung2008}]
\label{prop:IEV_Gram}
Let $M$ be a real Gram matrix. Then $\I(M)\subset\R^n$ and
$\#\I(M)\leq 2^n$. Moreover, the inverse eigenvectors of $M$ are
precisely the local maximizers of $|p|$ on the quadric
\[
\E_M=\{x\in\R^n:x^\top Mx=n\}.
\]
\end{proposition}
In particular, every positive real Gram matrix has exactly $2^n$ inverse eigenvectors, one in each intersection of the ellipsoid $\E_M$ with an orthant of $\R^n$. For completeness, we provide a proof in Section~\ref{section:IEV}.


The main result of the article provides an explicit decomposition of the identity operator in terms of rank-one tensors associated with the inverse eigenvectors of a matrix $M$ whose diagonal entries are equal to $1$. We establish this decomposition in two settings, beginning with the more general complex case. 

\begin{theorem}[Inverse Spectral Theorem for transverse and proper complex matrices] \label{thm:inversespectral_complex}
Let $M \in \C^{n \times n}$ be transverse and proper, and suppose that $\Diag M = I_n$. For each $\alpha\in\mathcal{I}(M)$, define
\begin{equation}\label{eq:cdef}
c(\alpha) = \frac{1}{ \det\bigl(I_n+\D(\alpha) M \D(\alpha)\bigr)}.
\end{equation}
Then 
\begin{equation}\label{eq:sumc}
\sum_{\mathcal{I}(M)} c(\alpha)=1,
\end{equation}
and
\begin{equation}\label{eq:iev_decomposition}
\sum_{\I(M)} c(\alpha)\,\alpha\otimes\alpha = I_n.
\end{equation}
\end{theorem}

In particular, by Lemma~\ref{lemma:dominant_transverse}, the above result applies for complex matrices $M$ with $\Diag(M)=I_n$ and $\| M - I_n \|_\infty < 1/2$.

In the proof of Theorem~\ref{thm:inversespectral_complex}, which is presented in Section~\ref{section:proof}, the properness of $M$ is essential for applying Proposition~\ref{prop:globres}, our main tool from residue theory. Although this assumption is not particularly restrictive -- the set of proper matrices is open and dense by Proposition~\ref{prop:Mproper} -- removing it is important for several applications. We are able to do so for real Gram matrices $M$ satisfying $m_{kk}=1$ and $|m_{kl}| <1$ for every $k \neq l$. The latter condition is equivalent to the strict positivity of all principal minors of order two. 

Under these assumptions,  the decomposition gives rise to a centered discrete isotropic probability measure on $\R^n$, reminiscent of measure associated with John's ellipsoid theorem \cite{john1948extremum, Ball1997}. In contrast to John's decomposition, however, the measure obtained by the present method is explicit and can be computed directly.

\begin{theorem}[Inverse Spectral Theorem for real correlation matrices] 
\label{thm:inversespectral_real}
Let $M \in \R^{n \times n}$ be a real correlation matrix satisfying $|m_{kl}| <1$ for all $k \neq l$. For each $\alpha\in\mathcal{I}(M)$, define $c(\alpha)$ by \eqref{eq:cdef}. Then 
$c(\alpha) > 0 $
for every $\alpha \in \I(M)$,
\[
\sum_{\I(M)} c(\alpha)=1,
\]
and
\[
\sum_{\I(M)} c(\alpha)\,\alpha\otimes\alpha = I_n.
\]
Equivalently, the measure 
\begin{equation*}\label{eq:mudef}
\mu_M = \sum_{\I(M)} c(\alpha) \delta_\alpha
\end{equation*}
is a centered discrete isotropic probability measure on $\R^n$.
\end{theorem}

Further equivalent formulations of Theorem~\ref{thm:inversespectral_real} assert that the identity operator belongs to the convex hull of the rank-one tensors associated with the inverse eigenvectors; the inverse eigenvectors form a weighted tight frame; their suitably scaled copies constitute a weighted spherical $2$-design; and they provide a Brascamp-Lieb datum.

Note that for positive definite $M$,  symmetry and the positivity of the corresponding principal minors of order $2$ ensure the criterion $|m_{kl}| <1$ for all $k \neq l$.
\medskip

In the nonsingular case, Theorem~\ref{thm:inversespectral_real} can also be obtained by the methods of Ouimet and Greaves \cite[Formula (2.4)]{OuimetGreaves2026}. The proof of that result given by Martínez and Ortega-Moreno also extends to the singular case, see \cite[Theorem 2.4]{MartinezOrtegaMoreno2026_Aomoto}.

\medskip
In both settings, multiplying \eqref{eq:iev_decomposition} by $M$ on the left yields a representation of $M$ in terms of its inverse eigenvectors:
\[
M = \sum_{\mathcal{I}(M)} c(\alpha)\,\alpha^{-1}\otimes\alpha.
\]
This identity also illustrates the necessity of the normalization condition $\Diag M = I_n$.

\medskip

We expect the Inverse Spectral Theorem to have wide-ranging applications. In Section~\ref{section:consequences}, we illustrate its strength by providing alternative proofs of several results, including multiple variants of the polarization problem, a sharp combinatorial form of Vaaler's theorem, various plank theorems, and a geometric result regarding ellipsoids.

\section{Residues of rational differential forms}\label{section:residues}

\noindent
In this section, we follow the works of Tsikh~\cite{Tsikh1992} and Cattani and Dickenstein~\cite{cattani2005introduction}.

Let $f_1, \ldots, f_n$ and $g$ be polynomials on $\C^n$. Set $f: \C^n \to \C^n$ to be the polynomial mapping $f=(f_1, \ldots, f_n)$. We will consider the meromorphic differential form
\begin{equation}\label{eq:meroform}
\omega = \frac {g \dd x_1 \wedge \cdots \wedge \dd x_n}{f_1 \cdots f_n} = \frac {g \dd x}{f}
\end{equation}
defined on $\C^n$, where $\dd x=\dd x_1\wedge\cdots\wedge \dd x_n$.

Let $D(f) = Z(f_1 \cdots f_n) = Z(f_1) \cup \ldots \cup Z(f_n)$ denote the {\em polar set} of $\omega$. Note the distinction between $D(f)$ and the zero set $Z(f)=Z(f_1) \cap \ldots \cap Z(f_n) $.

Assume that $\alpha  \in Z(f)$ is an isolated zero with a corresponding open neighborhood $U_\alpha$ that separates $\alpha$ from the further zeros of $f$. The {\em local (Grothendieck) residue} of the form $\omega$ at $\alpha$ is defined as
\begin{equation}\label{eq:residuedef}
    \res_{f, \alpha} (g) = \frac 1 {(2 \pi i )^n} \int_{\Gamma^\eps_\alpha(f)} \omega =  \frac 1 {(2 \pi i )^n} \int_{\Gamma^\eps_\alpha(f)} \frac {g \dd x}{f}
\end{equation}
over the closed oriented real $n$-cycle  $\Gamma^\eps_\alpha$ given by 
\begin{equation*}\label{}
\Gamma^\eps_\alpha (f) = \{ x \in U_\alpha: \ |f_j(x)|= \eps \textrm{ for every } j \in [n] \}
\end{equation*}
where $\eps >0$ is a sufficiently small positive number such that $\Gamma^\eps_\alpha(f) \subset U_\alpha$. The orientation of the $n$-cycle $\Gamma^\eps_\alpha(f)$ is determined by the condition that $\dd (\arg f_1)\wedge\cdots\wedge \dd (\arg f_n)  $ is positive on $\Gamma^\eps_\alpha(f)$.

Thus, $\Gamma^\eps_\alpha(f)$ is the local component inside $U_\alpha$ of the inverse image under $f$ of the torus
\[
{(w_1,\ldots,w_n)\in\C^n:\ |w_1|=\cdots=|w_n|=\varepsilon}.
\]
More generally, one may also allow distinct radii $\eps_1, \ldots, \eps_n >0$ for the components, although in most cases, using an identical $\eps$ will suffice for our purposes.

Note that for all sufficiently small values of $\eps$, the cycles $\Gamma^\eps_\alpha(f)$ are contained in $U_\alpha$ and  are homologous to each other in the domain $U_\alpha  \setminus D(f)$. Since $\omega$ is holomorphic here, Stokes' theorem implies that the value of the integral is independent of $\eps$ in this range.

If $\alpha$ is a simple zero of $f$, then $f$ admits a local holomorphic inverse, and Cauchy's formula (see e.g. \cite[Chapter~1]{Krantz1992}) implies that 
\begin{equation}\label{eq:res_local}
\res_{f, \alpha} (g)= \frac{g(\alpha)}{J_f(\alpha)}.
\end{equation}

From now on, we assume that $f_1, \ldots, f_n$ defines a zero-dimensional complete intersection; that is, $Z(f)$ is discrete and, hence, finite. We then define the {\em global residue} of $\omega$ as 
\begin{equation}\label{eq:res_global}
\Res_f(g) = \sum_{Z(f)} \res_{f, \alpha}(g)
\end{equation}
(cf.  \cite[Section 1.5.1]{cattani2005introduction}, \cite[\S II.5.3]{Tsikh1992}).

In order for the subsequent arguments, we will require that $f$ is proper; that is,  $f^{-1}(C)$ is compact in $\C^n$ for every compact $C \subset \C^n$. In the present setting, this property has a simple alternative description. The following statement is a special case of \cite[Lemma 1.7]{cattani_dickenstein_sturmfels}.

\begin{lemma}\label{lemma:fproper}
Assume that the polynomials $f_1, \ldots, f_n: \C^n \to \C$ are of the form
\begin{equation}\label{eq:fj_form}
    f_j = \tf_j + c_j
\end{equation}
where $\tf_j$ is a homogeneous polynomial of degree $d_j >0$ and $c_j \in \C$. The polynomial map $f = (f_1, \ldots, f_n)$ is proper if and only if $\tf = (\tf_1, \ldots, \tf_n)$  has no nontrivial zero in~$\C^n$; that is, $Z(\tf)=\{0\}$.
\end{lemma}

\begin{proof}
If $0 \neq z \in Z(\tf)$, then $f(w z) = (c_1, \ldots, c_n)$ for any $w \in \C$; therefore, $f$ is not proper. 

Conversely, assume that $\tf$ has no nontrivial zero. Then its minimum modulus on the unit sphere, say $\delta$, is strictly positive. Let $d = \min d_j$. Take any nonzero $x \in \C^n$, let $|x|=r$, and set $z = x/r$ that is a unit vector. Then 
\begin{equation*}\label{eq:tf}
|\tf(x)|^2 = |\tf(rz)|^2 = \sum_j |r^{d_j} \tf_j(z)|^2 \geq  r^{2d} \sum_j | \tf_j(z)|^2 = r^{2d} |\tf(z)|^2 \geq r^d \delta^2
\end{equation*}
whenever $r >1$. This ensures that $|\tf(x)| \to \infty $ as $|x| \to \infty$; consequently, $|f(x)| \to \infty$ as well.
\end{proof}

We note that the second part of the proof also implies the following, more general statement:

\begin{lemma}\label{lemma:fproper_weak}
Assume that $f=(f_1, \ldots, f_n): \C^n \to \C^n$ is a polynomial mapping with leading homogeneous part $\tf$.  If $Z(\tf ) = \{0\}$, then $f$ is proper.
\end{lemma}

\begin{proof}
For each $j \in [n]$, write
\[
f_j = \tf_j + h_j,
\]
where $\tf_j$ is a homogeneous polynomial of degree $d_j >0$, and $h_j$ is a polynomial with $\deg h_j < d_j$. For every $j$,
\[
\frac{h_j(rz)}{r^{d_j}}\to 0
\]
uniformly for  $z$ on the unit sphere as $r \to \infty$ through the positive reals. Consequently, 
\[
\lim_{r \to \infty} \frac{|f_j(r z)|}{r^{d_j}} = |\tf_j(z)|
\]
uniformly on the unit sphere. Since $|\tf|$ is bounded away from zero on the unit sphere by the condition $Z(\tf ) = \{0\}$, the remainder of the proof proceeds as in Lemma~\ref{lemma:fproper}.
\end{proof}

If $f$ is proper, then for any $r>0$, 
\begin{equation}\label{eq:gammarf}
\Gamma^r(f) = \{ x \in \C^n: \ |f_j(x)|= r,\  j\in [n] \}
\end{equation}
is  a compact real $n$-cycle. We orient $\Gamma^r(f)$ so  that $\dd (\arg f_1)\wedge\cdots\wedge \dd (\arg f_n)  $ is positive. As $r$ varies, these cycles are homologous in
$\C^n\setminus D(f)$; see \cite[\S II.8]{Tsikh1992}.  For sufficiently small $\eps>0$, the cycle $\Gamma^\eps(f)$ is the
disjoint union of local cycles contained in neighborhoods of the points
of $Z(f)$. Hence, for such $\eps$, integrating $\omega$ over
$\Gamma^\eps(f)$ yields the sum of the local residues by \eqref{eq:residuedef}, and therefore the
global residue defined in \eqref{eq:res_global}. On the other hand, since
the cycles $\Gamma^r(f)$ are homologous in $\C^n\setminus D(f)$ and $\omega$ is closed there, Stokes' theorem implies that the value of the
integral is independent of $r>0$. Therefore, the global residue of the meromorphic form $\omega$ in
\eqref{eq:meroform} can be computed by a single integral; see also \cite[Section 1.5.1]{cattani2005introduction}):
\begin{equation}\label{eq:Res_int}
\Res_f(g) = \frac 1 {(2 \pi i )^n} \int_{\Gamma^r(f)}  \frac {g \dd x}{f}
\end{equation}
where $r>0$ is arbitrary. This leads to an important simplification.

\begin{lemma}\label{lemma:globalresidue}
Assume that $f=(f_1, \ldots, f_n): \C^n \to \C^n$ is a polynomial map  with leading homogeneous part $\tf = (\tf_1, \ldots, \tf_n)$. Suppose that $\tf$ has no nontrivial zero. Set 
\begin{equation*}\label{eq:rho}
\rho:= \sum_j  d_j - n
\end{equation*}
where $d_j$ is the degree of $f_j$. Then for any polynomial $g: \C^n \to \C$ with  $\deg g \leq \rho$, 
\[
\Res_f(g) = \Res_{\tf}(g_\rho),
\]
where $g_\rho$ is the homogeneous part of $g$ of degree $\rho$.
\end{lemma}

\begin{proof}
By Lemma~\ref{lemma:fproper_weak}, both $f$ and $\tf$ are proper polynomial maps. As before, write 
\[
f_j = \tf_j + h_j,
\]
where $\tf_j$ is the leading homogeneous part of $f_j$, and $\deg h_j < \deg f_j$.

For $R>0$, set
\[
\widetilde{\Gamma}_R:= R\,\Gamma^1(\widetilde f)= \{x\in\mathbb C^n:\ |\widetilde f_j(x)|=R^{d_j},\ j\in[n]\}
\]
to be the compact real $n$-cycle  oriented so that $\dd (\arg \tf_1)\wedge\cdots\wedge \dd (\arg \tf_n)  $ is positive (compactness is ensured since $\tf$ is proper). If $R$ is sufficiently large, then $|\tf_j(x) |>|h_j(x)|$ holds on $\widetilde{\Gamma}_R$
for every $j \in [n]$ . Then, by \cite[\S II.8.1, Lemma]{Tsikh1992}, the cycle $\widetilde{\Gamma}_R$ is homologous in $\C^n \setminus D(f)$  to $\Gamma^{\delta}(f)$ for every sufficiently small $\delta>0$.  Accordingly, by~\eqref{eq:Res_int} and Stokes' theorem,
\begin{align*}
\Res_f(g) &= 
 \frac 1 {(2 \pi i )^n} \int_{\Gamma^\delta(f)}  \frac {g(x) \dd x}{f(x)} \\
&=\frac 1 {(2 \pi i )^n} \int_{\widetilde{\Gamma}_R}  \frac {g(x) \dd x}{f(x)}\\
&=\frac 1 {(2 \pi i )^n} \int_{\Gamma^1(\tf)}  \frac {g(Rx) R^n \dd x}{f_1(Rx)\cdots f_n(Rx)}\,.
\end{align*}
Letting $R \to \infty$,
\begin{equation}\label{eq:Reslim}
\Res_f(g) = \lim_{R\to \infty }\frac 1 {(2 \pi i )^n} \int_{\Gamma^1(\tf)}  \frac {(g(Rx)/R^\rho) \dd x}{(f_1(Rx)/R^{d_1})\cdots (f_n(Rx)/R^{d_n})}
\end{equation}
Here, 
\[
\lim_{R \to \infty} \frac{g(Rx)}{R^\rho} = g_\rho(x), \textrm{ and } \lim_{R \to \infty} \frac{f_j(Rx)}{R^{d_j}} = \tf_j(x) \textrm{ for every  } j\in [n].
\]
Moreover, the convergence is uniform on the compact cycle $\Gamma^1(\tf)$. Since $|\tf_j(x)|=1$ on  $\Gamma^1(\tf)$, the factors in the denominator of the integrand in \eqref{eq:Reslim} are uniformly separated away from zero. Therefore, the integrand converges uniformly, and the limit may be interchanged with the integral:
\begin{align*}
\Res_f(g)&= \frac 1 {(2 \pi i )^n} \int_{\Gamma^1(\tf)} \lim_{R\to \infty } \frac {(g(Rx)/R^\rho) \dd x}{(f_1(Rx)/R^{d_1})\cdots (f_n(Rx)/R^{d_n})} \\
&= \frac 1 {(2 \pi i )^n} \int_{\Gamma^1(\tf)} \frac {g_\rho(x)  \dd x}{\tf_1(x) \cdots \tf_n(x)}\\
&=\Res_{\tf}(g_\rho),
\end{align*}
where \eqref{eq:Res_int} is used at the last equality.
\end{proof}

If $f$ is transverse (i.e., 0 is a regular value of $f$), then the local residue formula \eqref{eq:res_local} can be applied at every $\alpha \in Z(f)$. Moreover, since the only zero of $\tf$ is at the origin, the global residue of the meromorphic form $ g_\rho  \dd x / \tf $ equals its local residue at~$0$. Consequently,  \eqref{eq:res_global} and Lemma~\ref{lemma:globalresidue} imply the following statement.

\begin{proposition}\label{prop:globres}
Assume that $f=(f_1, \ldots, f_n): \C^n \to \C^n$ is a transverse polynomial map with leading homogeneous part  $\tf=(\tf_1, \ldots, \tf_n)$ that has no nontrivial zero. Let $d_j = \deg f_j$ for each $j$, and define $\rho = \sum d_j - n$. Then for any polynomial $g$ with $\deg g \leq \rho$,
\[
\sum_{Z(f)} \frac{g(\alpha)}{J_f(\alpha)} = \res_{\tf, 0} (g_\rho),
\]
where  $g_\rho$ denotes the homogeneous component of $g$ of degree $\rho$. 
\end{proposition}

This is the key result from residue theory that we shall use. 
In particular, when $\deg g < \rho$, then $g_\rho = 0$, and hence the right hand side vanishes. Thus, we recover the Euler--Jacobi vanishing theorem (see e.g. \cite[\S II. 7.3., Corollary 2.]{Tsikh1992} or \cite[Theorem~1.5.12]{cattani2005introduction}). 

We note that in subsequent applications, the global residue in Proposition~\ref{prop:globres} can also be calculated directly by the method described in \cite{cattani_dickenstein_sturmfels}.

\section{Transversality and properness of matrices}\label{section:matrices}

In this section, we establish the required properties of the polynomials $f_{M,k}= x_k (M x)_k -1$ defined in \eqref{eq:fkdef} and the polynomial mapping $f_M=(f_{M,1}, \ldots, f_{M,n})$ given by \eqref{eq:fmdef}. Recall that  $J_M(x) = J_{f_M}(x)$ is the Jacobian of $f_M$ at $x$.

In order to apply Proposition~\ref{prop:globres} to $f_M$, we need to verify that $f_M$ possesses two properties: it is transverse, and proper. We characterize these attributes in terms of the matrix $M$. First, we aim for transversality. 

\begin{proof}[Proof of Proposition~\ref{prop:transverse}]
Note that 
\begin{equation}\label{eq:gradfk}
\nabla f_{M,k}(x) =  (Mx)_k \cdot e_k + x_k \cdot m_k,
\end{equation}
where $e_k$ is the $k$th standard basis vector, and $m_k \in \C^n$ is the $k$th row vector of $M$. Accordingly, the Jacobi matrix of $f_M$ at $x$ is given by
\[
\D(Mx) + \D(x) M.
\]
For an inverse eigenvector $\alpha \in \I(M)$, $M \alpha = \alpha^{-1}$ by \eqref{eq:iev_def}. Moreover, $\alpha$ has no zero coordinates, thus $p(\alpha) \neq 0$. Therefore, 
\begin{equation}\label{eq:J_M}
    J_M(\alpha) = \det(D(\alpha^{-1}) + \D(\alpha) M )) = \frac 1 {p(\alpha)} \det \bigl(I_n + \D(\alpha) M\D(\alpha) \bigr).
\end{equation}
The matrix $M$ is transverse if and only if this quantity is nonzero for every $\alpha \in \I(M)$, which is equivalent to the condition $ \det \bigl(I_n + \D(\alpha) M\D(\alpha) \bigr) \neq 0$.

Let $\sigma( \D(\alpha) M\D(\alpha)) = \{ \lambda_1, \ldots, \lambda_n \}$. Then the eigenvalues of $I_n + \D(\alpha) M\D(\alpha)$ are $1 + \lambda_j$, $j \in [n]$. Thus, 
\begin{equation}\label{eq:det_dmd}
\det \bigl(I_n + \D(\alpha) M\D(\alpha) \bigr) = \prod_j (1 + \lambda_j)
\end{equation}
which is zero iff $\lambda_j=-1$ for some $j$.
\end{proof}

Next, based on Lemma~\ref{lemma:fproper}, we characterize the cases when the mapping $f_M$ is proper. 

\begin{proof}[Proof of Proposition~\ref{prop:Mproper}]
Note that for each $k \in [n]$, $f_{M,k}$ is of the form \eqref{eq:fj_form} with $c_k=1$, $d_k =2$, and leading homogeneous part
\begin{equation}\label{eq:tfmj}
    \tf_{M,k}(x) = x_k (M x)_k.
\end{equation}
Therefore, by Lemma~\ref{lemma:fproper}, the map $f_M$ is proper if and only if
\begin{equation}\label{eq:tfm}
    \tf_M = (\tf_{M,1}, \ldots, \tf_{M,n})
\end{equation}
has no nontrivial zero. Assume that $x \in Z(\tf_M)$. Then for every $k$, 
\[
x_k (Mx)_k =0.
\]
Therefore, $x$ is a nontrivial zero if and only if $(Mx)_k=0$ for every $k$ for which $x_k \neq 0$. Such a zero exists if and only if the principal submatrix of $M$ corresponding to the nonzero coordinates of $x$ is singular; equivalently, the corresponding principal minor of $M$ is zero.
\end{proof}

In particular, we derive from Sylvester's criterion that every nonsingular (equivalently, strictly positive definite) real Gram matrix is proper. 

\medskip

It is a direct consequence of Proposition~\ref{prop:Mproper} that proper matrices form an open and dense set both in $\C^{n \times n}$ and $\R^{n \times n}$.

\begin{proof}[Proof of Lemma~\ref{lemma:realgram_transverse}]
Let $\alpha \in \I(M)$. Then by Lemma~\ref{lemma:realgram_realIEV}, whose proof is postponed to Section~\ref{section:IEV}, $\alpha \in \R^n$. Consequently, the matrix $\D(\alpha) M  \D(\alpha)$ is positive semidefinite, thus its eigenvalues are nonnegative real numbers. This ensures that $-1 \not \in \sigma(D(\alpha) M\D(\alpha))$.
\end{proof}

\begin{proof}[Proof of Lemma~\ref{lemma:dominant_transverse}]
To prove that $M$ is transverse, let $\alpha \in \I(M)$ be an inverse eigenvector of $M$. Let $a = \max_j |\alpha_j|$, and assume that for $l\in [n]$,
\[
a = |\alpha_l|.
\]
We will show that $0 \not \in \sigma(I_n +\D(\alpha) M\D(\alpha))$.  Note that the $k$th diagonal entry of this matrix is $1 +\alpha_k^2$, since $m_{kk}=1$. By Gershgorin's circle theorem, it suffices to verify that for every $k$, 
\[
|1 + \alpha_k|^2 > \sum_{j \neq k} |\alpha_k m_{kj} \alpha_j|.
\]
Note that 
\[
\sum_{j \neq k} |\alpha_k m_{kj} \alpha_j| \leq a^2 \sum_{j \neq k } |m_{kj}| < \frac{a^2}{2},
\]
thus it suffices to show that for every $k$, 
\begin{equation}\label{eq:ddom1}
|1 + \alpha_k|^2 \geq  \frac{a^2}{2}.
\end{equation}
Since $\alpha \in \I(M)$, for every $k$ we have 
\begin{equation}\label{eq:iev_a}
1 = \sum_j \alpha_k m_{kj} \alpha_j = \alpha_k^2 + \sum_{j \neq k } \alpha_k m_{kj} \alpha_j.
\end{equation}
Thus, 
\begin{equation}\label{eq:ddom2}
|1 + \alpha_k|^2 = \left|2 -\sum_{j \neq k } \alpha_k m_{kj} \alpha_j \right| \geq 2 - \sum_{j \neq k } |\alpha_k m_{kj} \alpha_j | \geq  2- a^2 \sum_{j \neq k } |m_{kj}| > 2 - \frac{a^2}{2}.
\end{equation}
On the other hand, by \eqref{eq:iev_a},
\[
a^2 = |\alpha_l|^2 = \left|1 -\sum_{j \neq l } \alpha_l m_{lj} \alpha_j \right| \leq 1 + a^2 \sum_{j \neq l } |m_{lj}| \leq 1 + \frac{a^2}{2}.
\]
It follows that $a^2 \leq 2$, and thus $2 - \frac{a^2}{2} \geq \frac{a^2}{2}$.
Hence, \eqref{eq:ddom2} implies \eqref{eq:ddom1}, completing the proof that $M$ is transverse.

To verify that $M$ is proper, note that by Gershgorin's circle theorem (or, equivalently, the L\'evy--Desplanques theorem
\cite[Theorem~1.21]{Varga2000}), every strictly diagonally
dominant matrix is nonsingular. Since every principal submatrix of $M$ is also strictly diagonally dominant, Proposition~\ref{prop:Mproper} yields that $M$ is proper. 
\end{proof}

\section{Inverse eigenvectors} \label{section:IEV}

In this section, we study the structure of inverse eigenvectors of real or complex matrices. First, we establish that  inverse eigenvectors of real Gram matrices are real. 

\begin{proof}[Proof of Proposition~\ref{lemma:realgram_realIEV}]
Suppose that $\alpha \in \C^n$ is an inverse eigenvector of $M$. Write $\alpha = a + i \cdot b$ with $a,b \in \R^n$. Note that for any nonzero complex number $z$, 
\[
\frac 1 z = \frac{\overline{z}}{|z|^2}.
\]
Since $M$ is real, we have $M \alpha = M a + i \cdot M b$. Thus, \eqref{eq:iev_def} implies that for every $k \in [n]$, 
\[
b_k (M b)_k = \frac{-b_k^2}{a_k^2 + b_k^2}.
\]
Summing over $k \in [n]$ implies that
\[
b^\top M b \leq 0.
\]
Since $M$ is positive semidefinite, equality must hold above. Consequently, $b_k =0$ for every $k$, verifying that $\alpha$ is real. 
\end{proof}

We continue with a useful approximation lemma. Recall that a matrix $M$ is transverse if $\I(M)$ is a finite set consisting of simple zeros of $f_M$. For such matrices $M$, the following useful continuity principle follows directly from the implicit function theorem \cite{krantz2002implicit} and from formula~\eqref{eq:IEV_locus}.

\begin{lemma}\label{lemma:IEV_persistence}
Suppose that $M \in \C^{n \times n}$ is transverse. Then the inverse eigenvectors of $M$ persist under small perturbations. That is, for every $\alpha \in \I(M)$, there exist open neighborhoods $U_\alpha \subset \C^n$ of $\alpha$ and $U_M \subset \C^{n \times n}$ of $M$ such that every $M'\in U_M$ has a unique inverse eigenvector $\alpha'$ in $U_\alpha$. In particular, $\alpha' \to \alpha$ as $M' \to M$.  Moreover, if $M_t \in \C^{n \times n}$ for $t>0$ with $M_t \to M$ as $t \to 0$, and $\alpha_t \in \I(M_t)$ converges to $\alpha \in \C^n$ as $t \to 0$, then $\alpha \in \I(M)$.
\end{lemma}

\begin{proof}
Since $\I(M) = Z(f_M)$ and $J_M(\alpha)\ne0$ for every $\alpha\in\I(M)$, the persistence property follows from the implicit function theorem.  Indeed, it yields a continuous local branch $\alpha' \in \I(M')$ with $\alpha'=\alpha$ when $M'=M$. Accordingly, $\alpha'\to\alpha$ as $M'\to M$. 

For the convergence assertion, note that the polynomial map $f_M$ depends on $M$ continuously; hence $f_M(\alpha) = \lim f_{M_t}(\alpha_t) =0$, and hence $\alpha\in Z(f_M)=\I(M)$.
\end{proof}

Newt, we describe the structure of inverse eigenvectors of real Gram matrices.

\begin{proof}[Proof of Proposition~\ref{prop:IEV_Gram}]
For this proof, we work in $\R^n$, which is allowed by Lemma~\ref{lemma:realgram_realIEV}. First, we show that every local maximizer of $|p(x)|$ on $\E_M$ is an inverse eigenvector of $M$.
Note that $p(x) =0$ on the coordinate hyperplanes, so no local maximizer can have a zero coordinate. Outside the coordinate hyperplanes, $p(x)$ is differentiable with 
\[
\nabla p(x) = p(x) \cdot x^{-1},
\]
and in a given orthant, $|p(x)|$ differs from $p(x)$ by  a constant sign factor. Since $M$ is symmetric,  $\nabla (x^\top M x) = 2 M x $. The Lagrange multiplier method implies that at every local maximizer $\alpha$, these two gradient vectors must be parallel: $\alpha^{-1} = \lambda M \alpha$ for some $\lambda \in \R$. Taking inner product with $\alpha$ on both sides yields
\[
n = \lambda \alpha^\top M \alpha
\]
which, combined with $\alpha \in \E_M$,  shows that $\lambda = 1$. Accordingly, $\alpha \in \I(M)$.

For the reverse direction, we need to show that every inverse eigenvector is a local maximizer. First, we  treat the case when $M$ is nonsingular. Let 
\[
E_M = \{x \in \R^n: x^\top M x \leq n \}
\]
be the solid ellipsoid with boundary $\E_M$. The homogeneity of $p$ implies that local maxima of $|p|$ in $E_M$ are attained on the boundary $\E_M$. Since $M$ is nonsingular, $E_M$ intersects all $2^n$ orthants of $\R^n$. Given that $E_M$ is compact, convex, and $\log p(x) = \sum \log x_j$ is strictly concave, there exists exactly one local maximizer of $|p(x)|$ in each intersection of $E_M$ with an orthant. This accounts for exactly $2^n$ local maxima, which also implies that $\# \I(M) \geq 2^n$. 

In order to prove that there are no further inverse eigenvectors, it suffices to show that $\# \I(M) \leq 2^n$. By definition, $\I(M)$ is the zero locus of the polynomial map $f = (f_1, \ldots, f_n)$. Since $f$ is proper by Proposition~\ref{prop:Mproper}, $Z(f) \subset \C^n$ is finite.  Hence, Bézout's theorem implies that $ \# \I(M) = \# Z(f) \leq \prod \deg (f_j) = 2^n$, completing the proof of the nonsingular case.  Note that we also derive that $ \# \I(M) = 2^n$.

Finally, assume that $M$ is singular. Since $M$ is symmetric and positive semidefinite, there exists a sequence of positive Gram matrices $M_t$ for $t>0$ that converge to $M$ as $t \to 0$. 
Let 
\[
\E_M^t = \{ x \in \R^n: x^\top M_t \, x = n \}.
\]
be the ellipsoid corresponding to $M_t$. Let $\alpha \in \I(M)$ be an inverse eigenvector of~$M$. Since $M$ is transverse by Lemma~\ref{lemma:realgram_transverse},  we may apply Lemma~\ref{lemma:IEV_persistence}, 
yielding that one can select $\alpha_t \in \I(M_t)$ for every $t >0$ such that $\alpha = \lim_{t \to 0} \alpha_t$. Note that $\E_M^t \to \E_M$ uniformly in a neighborhood of $\alpha$. Since each $\alpha_t$ is a local maximizer of $|p|$ on $\E_M^t$, this implies that $\alpha$ is a local maximizer of $|p|$ on $\E_M$, which completes the proof. 
\end{proof}

Note that when  $M$ is a singular Gram matrix, then in a given orthant, $|p|$ has a (unique) local maximizer if and only if $\ker M$ does not intersect the closed orthant (otherwise, the intersection with $\E_M$ is unbounded). Accordingly, orthants that avoid $\ker M$ contain precisely one element of $\I(M)$, while the remaining orthants contain no inverse eigenvectors of $M$.

 \section{Proof of the Inverse Spectral Theorem}\label{section:proof}

After the necessary preparations, we prove the main result of the paper. We start with a key lemma. 

\begin{lemma}\label{lemma:sumc}
Suppose that $M \in \C^{n \times n}$ is a transverse and proper matrix. For $\alpha \in \I(M)$, define
\[
c(\alpha) = \frac{1}{ \det\bigl(I_n+D(\alpha) M\D(\alpha)\bigr)}
\]
as \eqref{eq:cdef}. Then 
\begin{equation*}
    \sum_{\I(M)}c(\alpha)=1.
\end{equation*}
\end{lemma}

\begin{proof}
For $k \in [n]$, let $f_{M,k}: \C^n \to \C$ be the polynomial
\[
f_{M,k} = x_k (Mx)_k -1
\]
defined in \eqref{eq:fkdef}. Then $f_{M,k}$ has  degree $d_k = 2$ and leading homogeneous part
\[
\tf_{M,k}(x) = x_k (Mx)_k.
\]
Set $\rho = \sum d_k - n = n. $
Let $f_ M =(f_{M,1},\ldots,f_{M,n}): \C^n \to \C^n$ be the polynomial map as in \eqref{eq:fmdef}, 
and define $\tf_ M =(\tf_{M,1},\ldots,\tf_{M,n})$ to be the leading homogeneous part of $f_M$. Note that $Z(f_M) = \I(M)$ and $Z(\tf_M) = \{0\}$ since $f_M$ is proper.

Set 
\begin{equation}\label{eq:gmdef}
    g_M(x) = \prod_{j=1}^n (Mx)_j
\end{equation}
that is a degree $n$ homogeneous polynomial. Since $\deg g_M = n = \sum d_j - n$,  we may apply Proposition~\ref{prop:globres} for $f_M$ and $g_M$, yielding that
\begin{equation}\label{eq:sumc1}
\sum_{\I(M)} \frac{g_M(\alpha)}{J_M(\alpha)} = \res_{\tf_M, 0} (g_M).
\end{equation}
For any inverse eigenvector $\alpha \in \I(M)$, \eqref{eq:cdef}  and \eqref{eq:J_M} imply that
\begin{equation}\label{eq:jma}
J_M(\alpha) = \frac{1}{p(\alpha) c(\alpha)},
\end{equation}
while  it follows from \eqref{eq:iev_def} and \eqref{eq:gmdef} that
\[
g_M(\alpha) = \frac{1}{p(\alpha)}.
\]
Therefore, \eqref{eq:sumc1} simplifies to 
\[
\sum_{\I(M)} c(\alpha) =  \res_{\tf_M, 0} (g_M).
\]
To conclude the result, we will show that the residue on the right hand side is 1. 

We invoke the transformation formula for local residues \cite[\S II. 5.5., Proposition~1]{Tsikh1992}. This states the following. Let
$f=(f_1,\ldots,f_n)$ be a holomorphic mapping from a neighborhood of
$a\in\C^n$ to $\C^n$, and suppose that $a$ is an isolated zero of $f$. Let $A = (a_{kj}(x))$ be an $n \times n$ matrix whose elements are holomorphic functions of $x$ in a neighborhood of $a$, and define $F = A f$, that is, $F = (F_1, \ldots, F_n)$ with
\[
F_k = \sum_j a_{kj}(x) f_j(x).
\]
If $a$ is also an isolated zero of $F$, then, for every function $h:\C^n\to\C$ that is holomorphic in a neighborhood of $a$,
\begin{equation}\label{eq:transformation}
\res_{f,a}(h) = \res_{F,a} (h \cdot \det A).
\end{equation}

We apply this formula for the diagonal matrix $A(x) = \D(Mx)$. Then for every $x \in \C^n$,
\[
\tf_M(x) = A(x)  \id_n(x).
\]
The entries of $A(x)$ are entire functions on $\C^n$, and 
\[
\det A(x) =  \prod_{j} (Mx)_j=g_M(x).
\]
Setting $f = \id_n$, $F = \tf_M$, $A(x)=\D(Mx) $  and $h =1$ in \eqref{eq:transformation}, we derive that
\[
 \res_{\id_n,0}(1) = \res_{\tf_M, 0} (g_M).
\]
Since  $0$ is a simple zero of $I_n$ with  $J_{I_n}(0) = 1$, the residue $\res_{I_n,0}(1)$ can be computed by the local residue formula \eqref{eq:res_local}. Consequently,
\[
\sum_{\I(M)} c(\alpha) =  \res_{\tf_M, 0} (g_M)=\res_{I_n,0}(1) = \frac{1}{J_{I_n}(0)} = 1.
\qedhere
\]
\end{proof}

\begin{proof}[Proof of Theorem~\ref{thm:inversespectral_complex}]
Let $M \in \C^{n \times n}$ be a transverse and proper matrix. Note that Lemma~\ref{lemma:sumc} implies that $\sum_{\I(M)} c(\alpha) =1$,  thus proving~\eqref{eq:sumc}.

In order to derive \eqref{eq:iev_decomposition}, we will show that for every $k,l\in [n]$, 
\begin{equation}\label{eq:sumc_delta}
\sum_{\I(M)} c(\alpha) \alpha_k \alpha_l  = \delta_{kl}
\end{equation}
where $\delta_{kl}=1$ if $k=l$, and $0$ otherwise. 

First, suppose that $k \neq l$. Introduce the polynomial $g_{kl}: \C^n \to \C$ given by
\begin{equation}\label{eq:gkl}
g_{kl}(x) = \prod_{j \neq k,l} (Mx)_j.
\end{equation}
Then $g_{kl}$ is homogeneous of degree $n -2$. Accordingly, by Proposition~\ref{prop:globres} and formulae \eqref{eq:iev_def} and \eqref{eq:jma},
\begin{equation}\label{eq:sumckl}
 0 = \sum_{\I(M)} \frac{g_{kl}(\alpha)}{J_M(\alpha)} = \sum_{\I(M)} c(\alpha) \alpha_k \alpha_l 
\end{equation}
that verifies \eqref{eq:sumc_delta} for $k \neq l$.

Next, fix any $k \in [n]$. For every  $\alpha \in \I(M)$, \eqref{eq:iev_pol} implies that
\[
c(\alpha) = c(\alpha) \alpha_k \sum_j m_{kj} \alpha_j.
\]
Using Lemma~\ref{lemma:sumc}, \eqref{eq:sumckl} and the condition $m_{kk} =1$, summing the above equation for $\alpha \in \I(M)$ yields that
\begin{align*}
1&= \sum_{\I(M)} c(\alpha) \\
&= \sum_{\I(M)} c(\alpha) \alpha_k \sum_j m_{kj} \alpha_j  \\
&= \sum_{\I(M)} c(\alpha) m_{kk}\alpha_k^2 + \sum_{j \neq k} m_{kj}  \sum_{\I(M)} c(\alpha) \alpha_k \alpha_j \\
&=\sum_{\I(M)} c(\alpha) \alpha_k^2,
\end{align*}
completing the proof of \eqref{eq:sumc_delta}.
\end{proof}

Next, we extend the inverse spectral decomposition to real correlation matrices with strictly positive principal minors of order two. The proof proceeds by approximation with positive matrices, and the required convergence argument rests on the following lemma.

\begin{lemma}\label{lemma:convergence}
Let $M \in \R^{n \times n}$ be a real correlation matrix that satisfies $m_{kl}^2 \leq 1 - \delta$ for all $k \neq l$ with some positive $\delta >0$. Then for every $\alpha\in\mathcal{I}(M)$ with $|\alpha|^2 \geq 2n$,
\[
\det\bigl(I_n + \D(\alpha) M \D(\alpha)\bigr) \geq 1 + |\alpha|^2 + \frac{\delta}{8 n^4} |\alpha|^4.
\]
\end{lemma}

\begin{proof}
Let $\alpha=(\alpha_1, \ldots, \alpha_n) \in \R^n$ be an inverse eigenvector of $M$. Choose distinct indices $r,s \in [n]$ such that $|\alpha_r| \geq |\alpha_s| \geq |\alpha_j|$ for every $j \neq r,s$. Then 
\[
\alpha_r^2 \geq \frac{|\alpha|^2}{n} \geq 2.
\]
Furthermore, since $m_{rr}=1$ and $\alpha \in \I(M)$, \eqref{eq:iev_pol} implies that
\[
\alpha_r^2 + \alpha_r \sum_{j \neq r} m_{r j} \alpha_j =1.
\]
Thus, 
\[
n \cdot |\alpha_s| > \Big | \sum_{j \neq r} m_{r j}\alpha_j \Big | =  \frac{\alpha_r^2 -1}{| \alpha_r |} \geq \frac {|\alpha_r|}{2} \geq  \frac{|\alpha|}{2 \sqrt{n}},
\]
leading to
\[
\alpha_s^2 \geq \frac{|\alpha|^2}{4 n^3}.
\]
Set $A= \D(\alpha) M \D(\alpha)$. Since $M$ is positive semidefinite, so is $A$. Hence its eigenvalues $\lambda_1,\ldots,\lambda_n$ are nonnegative, and
\begin{equation}
\det(I_n+A)
\ge
1+\sum_j\lambda_j+\sum_{k<\ell}\lambda_k\lambda_\ell.
\end{equation}
Moreover,
\[
\sum_j \lambda_j = \tr A = |\alpha|^2.
\]
Also,
\begin{align*}
\sum_j \lambda_j^2 &= \tr A^2 = \|A\|_{HS}^2 \\
&= \sum_{k,l} \alpha_k^2 m_{kl}^2 \alpha_l^2 \\
&= \sum_k \alpha_k^4 +  \sum_{k \neq l} \alpha_k^2 m_{kl}^2 \alpha_l^2\\
&\leq \sum_k \alpha_k^4 + (1 - \delta) \sum_{k \neq l} \alpha_k^2 \alpha_l^2 \\
&= |\alpha|^4 - \delta \sum_{k \neq l} \alpha_k^2 \alpha_l^2 \\
& \leq  |\alpha|^4 - \delta \, \alpha_r^2 \alpha_s^2 \\
& \leq  |\alpha|^4 - \delta \, \frac{|\alpha|^4}{4 n^4}.
\end{align*}
Therefore,
\begin{align*}
\det\bigl(I_n + \D(\alpha) M \D(\alpha)\bigr)  &= \prod_j (1 + \lambda_j) \\
&\geq 1 + \sum_j \lambda_j + \sum_{k < l } \lambda_k \lambda_l  \\
&= 1 + |\alpha|^2 + \frac 1 2 \Big[ \Big( \sum_j \lambda_j \Big)^2 - \sum_j \lambda_j^2  \Big]\\
& \geq  1 + |\alpha|^2 +\delta \, \frac{|\alpha|^4}{8 n^4}.
\qedhere
\end{align*}
\end{proof}

\begin{proof}[Proof of Theorem~\ref{thm:inversespectral_real}]
Let $M$ be a real correlation matrix that satisfies $m_{kl}^2 <1 -\delta$ for every $k \neq l$ with some $\delta >0$.

By Proposition~\ref{prop:IEV_Gram}, $\I(M) \subset \R^n$. Moreover, for each inverse eigenvector $\alpha \in \I(M)$, the matrix $\D(\alpha) M \D(\alpha)$ is positive semidefinite. Accordingly, the eigenvalues of the matrix $I_n + \D(\alpha) M \D(\alpha)$ are positive real numbers, which implies that $c(\alpha) > 0$ by \eqref{eq:det_dmd}.

Next, choose pairwise disjoint, bounded neighborhoods $U(\alpha) \subset \R^n$ of each $\alpha \in \I(M)$. Since $\I(M)$ is finite, for every sufficiently large $R>0$,
\[
\bigcup_{\I(M)}U(\alpha) \subset B_R=\{ x \in \R^n: |x| \leq R\}.
\]
Let $M_t$, for $t>0$, be positive real correlation matrices such that $M_t\to M$ as $t \to 0$. For instance, one may take
\[
M_t=(1-t)M+tI_n
\]
for $0<t<1$.

Lemma~\ref{lemma:IEV_persistence} guarantees that for small enough $t$, each neighborhood $U(\alpha)$ contains precisely one element of $\I(M_t)$. Define
\[
\LL_t = \I(M_t) \setminus \bigcup_{\I(M)} U(\alpha). 
\]
Lemma~\ref{lemma:IEV_persistence} also implies that the limit of every convergent sequence $\beta_1, \beta_2,  \ldots$ with $\beta_j \in \I(M_{t_j})$ for some $t_1 > t_2 > \ldots >0$ must be an element of $\I(M)$. Therefore, by compactness, we derive that for sufficiently small $t$, 
\begin{equation}\label{eq:ltbr}
\LL_t \cap B_R = \emptyset.
\end{equation}
From now on, we assume that $t$ is small enough so that the above properties hold.

Since the matrix $M_t$ is transverse and proper by Lemma~\ref{lemma:realgram_transverse} and Proposition~\ref{prop:Mproper}, Theorem~\ref{thm:inversespectral_complex} implies that 
\begin{equation}\label{eq:sumct}
\sum_{\I(M_t)} c(\alpha_t) = 1
\end{equation}
and
\begin{equation}\label{eq:sumctaa}
\sum_{\mathcal{I}(M_t)} c(\alpha_t)\,\alpha_t\otimes\alpha_t = I_n.
\end{equation}
Moreover, $c(\alpha_t) >0$ for every $\alpha_t \in \I(M_t)$ since $M_t$ is positive definite.

Fix $\alpha \in \I(M)$, and   let $\alpha_t$ be the unique element of $\I(M_t) \cap U(\alpha)$.  Then by continutiy,
\begin{equation}\label{eq:clim}
c(\alpha_t) \to c(\alpha)
\end{equation}
and 
\begin{equation}\label{eq:caalim}
c(\alpha_t) \,\alpha_t\otimes\alpha_t \to c(\alpha)\,\alpha\otimes\alpha
\end{equation}
as $t \to 0$. 

Note that for $x \in \R^n$,
\[
\| x \otimes x \|_{HS} = |x|^2.
\]
Let $\beta_t \in \LL_t$ be arbitrary. Then $|\beta_t|>R$ by \eqref{eq:ltbr}. Accordingly,  Lemma~\ref{lemma:convergence} implies that
\begin{align*}
c(\beta_t)  \leq  \Big( 1 + |\beta_t|^2 + \frac{\delta}{8 n^4} |\beta_t|^4\Big)^{-1} < \frac{8 n^4}{\delta} R^{-4}
\end{align*}
and
\begin{align*}
\| c(\beta_t) \,\beta_t\otimes\beta_t \|_{HS} &\leq  |\beta_t|^2  \Big( 1 + |\beta_t|^2 + \frac{\delta}{8 n^4} |\beta_t|^4\Big)^{-1} \\
&= \Big( \frac 1 {|\beta_t|^2} + 1 + \frac{\delta}{8 n^4} |\beta_t|^2\Big)^{-1}\\
& < \frac{8 n^4}{\delta} R^{-2}.
\end{align*}
Therefore, using that $\#\I(M_t) \leq 2^n$,
\[
\sum_{\LL_t} c(\beta_t) < 2^n \cdot \frac{8 n^4}{\delta} R^{-4}
\]
and 
\[
\Big\| \sum_{\LL_t} c(\beta_t) \,\beta_t\otimes\beta_t \Big\|_{HS} < 2^n \cdot \frac{8 n^4}{\delta} R^{-2}.
\]

Taking $R \to \infty$ shows that as $t \to 0$, the contribution of the terms in \eqref{eq:sumct} and \eqref{eq:sumctaa} corresponding to elements of $\LL_t$ converges to 0 in absolute value, and in the Hilbert-Schmidt norm, respectively. Combined with the limit formulae \eqref{eq:clim} and \eqref{eq:caalim}, this verifies that the  identities \eqref{eq:sumc} and \eqref{eq:iev_decomposition} hold for $M$.

Finally, note that $\I(M)$ is centrally symmetric and $c(-\alpha)=c(\alpha)$, so
the corresponding probability measure is centered. This completes the proof of Theorem~\ref{thm:inversespectral_real}.
\end{proof}

\section{Extensions}\label{section:extensions}

In this section, we establish the Inverse Spectral Theorem for real Gram matrices whose diagonal entries are not necessarily equal to $1$. To this end, there are two possible approaches. The first one follows from a simple rescaling argument.

\begin{theorem}[Inverse spectral theorem for general Gram matrices]
\label{thm:inversespectral_general_IEV}
Let $M \in \R^{n \times n}$ be a real Gram matrix satisfying $|m_{kl}| <\sqrt{m_{kk}m_{ll}}$ for all $k \neq l$. For each $\alpha\in\mathcal{I}(M)$, define $c(\alpha)$ by~\eqref{eq:cdef}. 
Then $c(\alpha) > 0$ for every $\alpha \in \I(M)$,
\[
\sum_{\mathcal{I}(M)} c(\alpha)=1,
\]
and
\begin{equation}\label{eq:iev_decomposition2}
\sum_{\mathcal{I}(M)} c(\alpha)\,\alpha\otimes\alpha = \Diag(M)^{-1}.
\end{equation}
Equivalently, \begin{equation*}\label{eq:mudef2}
\mu_M = \sum_{\I(M)} c(\alpha) \delta_\alpha
\end{equation*}
is a centered discrete probability measure on $\R^n$ whose covariance matrix is
\begin{equation}
\int x x^\top \dd \mu_M = \Diag(M)^{-1}.
\end{equation}
\end{theorem}

\begin{proof}
If $M$ is positive definite, its diagonal entries are strictly positive. In case $M$ is singular, the criterion $|m_{kl}| <\sqrt{m_{kk}m_{ll}}$ rules out the possibility $m_{kk} = 0$ for some $k \in [n]$. 

Accordingly, the matrix 
\[
\Delta=\Diag(M),
\]
is positive definite. Let $N=\Delta^{-1/2}M\Delta^{-1/2}$. Then $N$ is a real correlation matrix that satisfies $|n_{kl}|<1$ for all $k \neq l$. Therefore, Theorem~\ref{thm:inversespectral_real} applies to $N$.

Note that the map $\alpha \mapsto \beta:=\Delta^{1/2}\alpha$ gives a bijection between $\I(M)$ and $\I(N)$. Indeed, for every $\alpha \in \I(M)$,
\[
N \beta = \Delta^{-1/2}M\Delta^{-1/2} \cdot \Delta^{1/2} \alpha = \Delta^{-1/2}M \alpha = \Delta^{-1/2} \alpha^{-1} = \beta^{-1}.
\]
The converse follows by reversing the same calculation. Furthermore, 
\[ \D(\beta)N\D(\beta)=\D(\alpha)M\D(\alpha), 
\]
so the corresponding weights are equal.

Applying Theorem~\ref{thm:inversespectral_real} to $N$, we therefore obtain that $c(\alpha) > 0$ for every $\alpha \in \I(M)$, and  $\sum_{\mathcal{I}(M)} c(\alpha)=1$. Moreover,

\[ I_n =  \sum_{\I(N)}c(\beta)\,\beta\otimes\beta
= \Delta^{1/2}
\Bigg(
\sum_{\I(M)}
c(\alpha)\,\alpha\otimes\alpha \Bigg) \Delta^{1/2}.
\]
Multiplying on the left and on the right by $\Delta^{-1/2}$ gives
\[
\sum_{\I(M)}
c(\alpha)\,\alpha\otimes\alpha
=
\Delta^{-1}
=
\Diag(M)^{-1}.
\qedhere
\]
\end{proof}

The second generalization follows by changing the polynomial equations \eqref{eq:iev_pol} that define inverse eigenvectors. For a real Gram matrix $M$, 
define the set of {\em diagonal inverse eigenvectors} by
\begin{equation}\label{eq:IDdef}
\I_\Delta(M) =
\left\{
\alpha\in\C^n:
\alpha_k(M\alpha)_k=m_{kk}
\text{ for every }k\in[n]
\right\}.
\end{equation}
Furthermore, set 
\[
\Delta=\Diag(M)
\]
as before.

\begin{theorem}[Inverse Spectral Theorem for diagonal inverse eigenvectors]
\label{thm:inversespectral_Diev}
Let $M\in\R^{n\times n}$ be a real Gram matrix satisfying $|m_{kl}|<\sqrt{m_{kk}m_{ll}}$
for all $k\neq l$. For $\alpha\in\I_\Delta(M)$, define
\[
c_\Delta(\alpha)
=
\frac{\det\Delta}
{\det\bigl(\Delta+\D(\alpha)M\D(\alpha)\bigr)} \, .
\]
Then $\I_\Delta(M)\subset\R^n$, $c_\Delta(\alpha)>0$
for every $\alpha\in\I_\Delta(M)$,
\[
\sum_{I_\Delta(M)}c_\Delta(\alpha)=1,
\]
and
\[
\sum_{I_\Delta(M)}
c_\Delta(\alpha)\,\alpha\otimes\alpha
=
I_n.
\]
Equivalently,
\[
\mu_{\Delta }
=
\sum_{I_\Delta(M)}
c_\Delta(\alpha)\delta_\alpha
\]
is a centered discrete isotropic probability measure on $\R^n$.
\end{theorem}

\begin{proof}
The argument is parallel to that of Theorem~\ref{thm:inversespectral_real}, with  the r\^ole of $f_{M,k}$  in \eqref{eq:fkdef} played by
\[
f^{\Delta}_{M,k}(x) = x_k(Mx)_k -m_{kk}.
\]
Then $\I_\Delta(M) = Z(f^\Delta_M)$ for the polynomial mapping 
\[
f^\Delta_M = \bigl(f^{\Delta}_{M,1}, \ldots,f^{\Delta}_{M,n}).
\]
A direct computation shows that for each $\alpha \in \I_\Delta(M)$,
\[
J_{f^\Delta_M}(\alpha)
=
\frac{
\det\bigl(\Delta+\D(\alpha)M\D(\alpha)\bigr)
}{p(\alpha)} \,.
\]

First, assume that $M$ is positive definite.  Then $f^\Delta_M$ is proper and transverse.
Let $g_M(x)=\prod_{j=1}^n(Mx)_j$ as in \eqref{eq:gmdef}. Then, for every  $\alpha \in \I_\Delta(M)$,
\[ g_M(\alpha)=\frac{\det\Delta}{p(\alpha)}, \]
hence 
\[ \frac{g_M(\alpha)}{J_{f^\Delta_M}(\alpha)} = c_\Delta(\alpha). \]
The same residue calculation as in Lemma~\ref{lemma:sumc} thus gives 
\begin{equation}\label{eq:cdsum}
    \sum_{I_\Delta(M)}c_\Delta(\alpha)=1. 
\end{equation}

For distinct $k,l\in[n]$, we apply Proposition~\ref{prop:globres} to $f^\Delta_M$ and $g_{kl}$ defined in \eqref{eq:gkl} which yields
\[
0 =
\frac{1}{m_{kk}m_{ll}}
\sum_{I_\Delta(M)}
c_\Delta(\alpha)\alpha_k\alpha_l.
\]
Thus all off-diagonal entries of
\[
\sum_{I_\Delta(M)}
c_\Delta(\alpha)\,\alpha\otimes\alpha
\]
vanish. On the other hand, using \eqref{eq:IDdef} and \eqref{eq:cdsum}, we derive that all diagonal entries
are equal to one. Hence
\[
\sum_{I_\Delta(M)}
c_\Delta(\alpha)\,\alpha\otimes\alpha=I_n.
\]

As in the proof of Theorem~\ref{thm:inversespectral_real}, one obtains $\I_\Delta(M)\subset\R^n$. 
Since $\Delta$ is positive definite and $\D(\alpha)M\D(\alpha)$ is positive semidefinite, it follows that $c_\Delta(\alpha)>0$ for every $\alpha\in\I_\Delta(M)$. 

The singular case follows by the same approximation argument as in the
proof of Theorem~\ref{thm:inversespectral_real}, using $M_t=(1-t)M+t\Delta.$
\end{proof}

We conclude the section with another generalization that is a consequence of the previous result.

\begin{theorem}[Weighted inverse spectral decomposition]
\label{thm:weighted_decomposition}
Let $M\in\R^{n\times n}$ be a real correlation matrix satisfying $|m_{kl}|<1$ for every $k \neq l$. Suppose that $t = (t_1,\ldots,t_n) \in \R^n$ with  $t_j>0$ for every $j \in [n]$, and let
\[ 
T=\D(t).
\]
Define
\[
\I_t(M) =
\left\{
\beta\in\R^n: \beta_k(M\beta)_k=t_k
\text{ for every }k\in[n]
\right\},
\]
and for each $\beta \in \I_t(M)$, introduce
\[
c_t(\beta) = \frac{\det T}
{\det\bigl(T+\D(\beta)M\D(\beta)\bigr)}.
\]
Then $c_t(\beta) > 0$ for every $\beta\in\I_t(M)$, 
\[
\sum_{\I_t(M)}c_t(\beta)=1
\]
and
\[
\sum_{\I_t(M)}
c_t(\beta)\,\beta\otimes\beta=T.
\]
\end{theorem}

\begin{proof}
Set
\[
G=T^{1/2}MT^{1/2}.
\]
Then $G$ is a real Gram matrix with $\Diag(G)=T$ and $|g_{kl}|<\sqrt{g_{kk}g_{ll}}$ for each $k \neq l$. Theorem~\ref{thm:inversespectral_Diev} therefore gives that $c_\Delta(\alpha)>0$ for every $\alpha \in \I_\Delta(G)$,
\[
\sum_{\I_\Delta(G)}
c_\Delta(\alpha)=1
\]
and
\[
\sum_{\I_\Delta(G)}
c_\Delta(\alpha)\,\alpha\otimes\alpha=I_n.
\]

For $\alpha\in\I_\Delta(G)$, set
\[
\beta=T^{1/2}\alpha.
\]
Since
\[
G\alpha=T^{1/2}M\beta,
\]
we have that for every $k \in [n]$,
\[
\alpha_k(G\alpha)_k = \beta_k(M\beta)_k.
\]
Thus, the map $\alpha\mapsto\beta:=T^{1/2}\alpha$ is a bijection
between $\I_\Delta(G)$ and $\I_t(M)$. Moreover,
\[
\D(\alpha)G\D(\alpha) = \D(\beta)M\D(\beta),
\]
so the corresponding coefficients satisfy
\[
c_\Delta(\alpha) = \frac{\det T} {\det\bigl(T+\D(\beta)M\D(\beta)\bigr)}
= c_t(\beta).
\]
Accordingly, $c_t(\beta) >0$ and $\sum_{\I_t(M)}c_t(\beta)=1$.
Finally, taking the congruence of the isotropic decomposition by
$T^{1/2}$ gives
\[
\sum_{\I_t(M)} c_t(\beta)\,\beta\otimes\beta = T^{1/2}I_nT^{1/2}=T.
\qedhere
\]
\end{proof}

\section{Consequences}\label{section:consequences}
We present a list of applications of Theorem~\ref{thm:inversespectral_real} and its generalizations presemted in the previous section. 
\begin{corollary}\label{cor_iev_norm}
Every real correlation matrix has an inverse eigenvector of norm at most~$\sqrt{n}$.
\end{corollary}
\begin{proof}
If $|m_{kl}| < 1$ for every $k \neq l$, then by taking traces in \eqref{eq:iev_decomposition},
\[
\sum_{\alpha \in \I(M)} c(\alpha)|\alpha|^2=n.
\]
As the coefficients $c(\alpha)$ are positive and sum to one, there exists $\alpha\in\I(M)$ such that 
\[ |\alpha|^2\leq n. 
\] 

For an arbitrary real correlation matrix $M$, apply the preceding argument to \[ M_t=(1-t)M+tI_n. \] Choose $\alpha_t\in\I(M_t)$ satisfying $|\alpha_t|\leq\sqrt n$. Passing to a convergent subsequence and using Lemma~\ref{lemma:IEV_persistence}, we obtain an inverse eigenvector $\alpha\in\I(M)$ with $|\alpha|\leq\sqrt n$.
\end{proof}

The arithmetic--geometric mean inequality readily implies
\begin{corollary} \label{cor_iev_prod}
For every real correlation matrix $M$, there exists an inverse eigenvector $\alpha  \in \I(M)$ satisfying $|p(\alpha)| \leq 1$.
\end{corollary}

These properties, in turn, lead to solutions to the polarization problems that have been recently proven by Martínez and Ortega-Moreno \cite{MartinezOrtegaMoreno2026}, and by Ouimet and Greaves~\cite{OuimetGreaves2026} via a stochastic approach. 

\begin{corollary}[The strong polarization inequality] \label{cor:strongpol}
Let $u_1, \ldots, u_n \in \R^d$ be real unit vectors. Then there exists a unit vector $v \in \R^d$ that satisfies
\begin{equation*}\label{eq:strongpol}
\sum_j \frac {1}{\la u_j, v \ra^2} \leq n^2.
\end{equation*}
\end{corollary}

\begin{proof}
Let $M$ be the Gram matrix of $u_1,\ldots,u_n$. By
Corollary~\ref{cor_iev_norm}, there exists $\alpha\in\I(M)$ such that
$|\alpha|^2\leq n$. Set
\[
w=\sum_j\alpha_j u_j.
\]
Since $M\alpha=\alpha^{-1}$, we have
\[
\la u_j,w\ra=(M\alpha)_j=\frac1{\alpha_j}
\]
for every $j\in[n]$. Moreover,
\[
|w|^2 = \alpha^\top M\alpha
= \sum_{k=1}^n\alpha_j(M\alpha)_j = n.
\]
Thus, $v=w/\sqrt n$ is a unit vector and
\begin{equation}\label{eq:ukv}
\la u_j,v\ra=\frac1{\sqrt n\,\alpha_j}.
\end{equation}
Consequently,
\[
\sum_j\frac1{\la u_j,v\ra^2}
= n\sum_j\alpha_j^2
= n|\alpha|^2 \leq n^2.
\qedhere
\]
\end{proof}

\begin{corollary}[The $n$th real linear polarization inequality] \label{cor:pol}
Let $u_1, \ldots, u_n \in \R^d$ be real unit vectors. Then there exists a unit vector $v \in \R^d$ that satisfies
\begin{equation*}\label{eq:prodpol}
\prod_{j} |\la u_j, v \ra| \geq n^{-n/2}.
\end{equation*}
\end{corollary}

\begin{proof}
Let $M$ be the Gram matrix of $u_1,\ldots,u_n$. By
Corollary~\ref{cor_iev_prod}, there exists $\alpha\in\I(M)$ such that
$|p(\alpha)|\leq1$. As above, set 
\[
w=\sum_j\alpha_j u_j
\]
and let $v= w /\sqrt n.$ Then $|v|=1$ and by \eqref{eq:ukv}, 
\[
\prod_j|\la u_j,v\ra|
=
\frac{n^{-n/2}}{|p(\alpha)|}
\geq n^{-n/2}.
\qedhere
\]
\end{proof}
 Galicer, Ortega-Moreno and Pinasco \cite{GalicerOrtegaMorenoPinasco2026} obtained the following generalization of Corollary~\ref{cor:strongpol}.
 
\begin{corollary}[Weighted strong polarization inequality \cite{GalicerOrtegaMorenoPinasco2026}]
\label{cor:weighted_strongpol}
Let $u_1,\ldots,u_n\in\R^d$ be unit vectors, and let
$t_1,\ldots,t_n>0$ satisfy $\sum t_j=1.$
Then there exists a unit vector $v\in\R^d$ such that
\[
\sum_j \frac{t_j^2}{\la u_j,v\ra^2} \leq 1.
\]
\end{corollary}

\begin{proof}
By approximation, it is enough to consider the case in which the
vectors $u_1,\ldots,u_n$ are pairwise nonparallel. Let $M$ be their
Gram matrix and set
\[
T=\D(t_1,\ldots,t_n).
\]
By Theorem~\ref{thm:weighted_decomposition},
\[
\sum_{\I_t(M)}
c_t(\beta)\,\beta\otimes\beta=T.
\]
Taking traces leads to
\[
\sum_{\I_t(M)}
c_t(\beta)|\beta|^2
=
\tr T
=
1.
\]
Since the coefficients are positive and sum to one, there exists
$\beta\in\I_t(M)$ such that
\[
|\beta|^2\leq1.
\]

Set
\[
v=\sum_j \beta_j u_j.
\]
Since $\beta_j(M\beta)_j=t_j$ for every $j \in [n]$, we have
\[
|v|^2 = \beta^\top M\beta
= \sum_j \beta_j(M\beta)_j
= \sum_j t_j = 1.
\]
Thus $v$ is a unit vector. Moreover, for every $j$,
\[
\la u_j,v\ra
=
(M\beta)_j
=
\frac{t_j}{\beta_j}.
\]
Consequently,
\[
\sum_j \frac{t_j^2}{\la u_j,v\ra^2} 
= \sum_j \beta_j^2
= |\beta|^2 \leq1.
\qedhere
\]
\end{proof}

Applying the weighted arithmetic--geometric mean inequality as in \cite{GalicerOrtegaMorenoPinasco2026} leads to the following result. 
\begin{corollary}[Weighted polarization inequality \cite{GalicerOrtegaMorenoPinasco2026}]
\label{cor:weighted_pol}
Under the assumtions of Corollary~\ref{cor:weighted_strongpol}, there exists a unit vector $v\in\R^d$ such that
\[
\prod_j |\la u_j,v\ra|^{t_j} \geq \prod_j t_j^{t_j/2} \,.
\]
\end{corollary}

We continue with another extension of the strong polarization inequality. We recover Corollary~\ref{cor:strongpol} by taking $A = I_n$ below. 

\begin{corollary}[Generalized matrix polarization inequality]
\label{cor:matrix_polarization}
Let $u_1,\ldots,u_n\in\R^d$ be unit vectors, and let
$A=(a_{kl})_{k,l=1}^n\in\R^{n\times n}$ be a positive semidefinite matrix.
Then there exists a unit vector $v\in\R^d$ such that
\[
\sum_{k,l}
\frac{a_{kl}}
{\la u_k,v\ra\la u_l,v\ra}
\leq
n\tr A.
\]
\end{corollary}

\begin{proof}
By a standard approximation argument, we may assume that the $u_i$ are pairwise nonparallel.
Let $M$ be the Gram matrix of $u_1,\ldots,u_n$. Apply Theorem~\ref{thm:inversespectral_real} to $M$ to yield the inverse spectral decomposition
\[
\sum_{\I(M)} c(\alpha)\,\alpha\otimes\alpha = I_n.
\]
Taking the Hilbert--Schmidt inner product of both sides with $A$ gives
\[
\sum_{\I(M)}
c(\alpha)\,
\la A\alpha,\alpha\ra
=
\tr A.
\]
Since the coefficients $c(\alpha)$ are positive and sum to $1$, there
exists $\alpha\in\I(M)$ such that
\[
\la A\alpha,\alpha\ra\leq\tr A.
\]
Set, as usual, $w=\sum_j\alpha_j u_j$. Then $v=w/\sqrt n$ is a unit vector and, by \eqref{eq:ukv},

\begin{align*}
\sum_{k,l}
\frac{a_{kl}}
{\la u_k,v\ra\la u_l,v\ra}
&=
n\sum_{k,l} a_{kl}\alpha_k\alpha_l\\
&=
n\la A\alpha,\alpha\ra\\
&\leq
n\tr A,
\end{align*}
which proves the assertion.
\end{proof}

Next, we establish a geometric consequence concerning the ellipsoid $\E_M$ introduced in Proposition~\ref{prop:IEV_Gram}.

\begin{corollary}
\label{cor:hyperboloid_trace}
Let $M\in\R^{n\times n}$ be positive definite, and suppose that the
ellipsoid
\[
\E_M=\{x\in\R^n:x^\top Mx=n\}
\]
meets every branch of the hypersurface $|p(x)|=1.$  Then
\[
\sum_{k=1}^n\frac1{m_{kk}}
\geq n.
\]
Consequently,
\[
\tr(M^{-1})\geq n.
\]
Moreover, equality holds if and only if $M=I_n$.
\end{corollary}

\begin{proof}
We may assume that $M$ is symmetric, hence a positive Gram matrix. 

For every orthant, let $\alpha\in\I(M)$ be the corresponding maximizer
of $|p|$ on $\E_M$. Since $\E_M$ meets the branch of $|p(x)|=1$ in that
orthant, we have
\[
|p(\alpha)|\geq1.
\]
By the arithmetic--geometric mean inequality,
\[
\frac{|\alpha|^2}{n}
\geq
|p(\alpha)|^{2/n}
\geq1,
\]
and hence $|\alpha|^2\geq n$.

We apply Theorem~\ref{thm:inversespectral_general_IEV} to $M$, yielding that 
\[
\sum_{\I(M)}
c(\alpha)\,\alpha\otimes\alpha
=
\Diag(M)^{-1}.
\]
By taking traces, we derive that 
\[
\tr\bigl(\Diag(M)^{-1}\bigr)
=
\sum_{\I(M)}c(\alpha)|\alpha|^2
\geq
n\sum_{\I(M)}c(\alpha)
=
n.
\]
Finally,
\[
(M^{-1})_{kk}\geq\frac1{m_{kk}}
\]
for every $k\in[n]$, and thus
\[
\tr(M^{-1})
\geq
\tr\bigl(\Diag(M)^{-1}\bigr)
\geq n.
\qedhere
\]
\end{proof}

\section{Acknowledgements}

The author thanks Professor Keith Ball for his continued guidance during the author’s doctoral studies at UCL, and Péter Frenkel, Máté Matolcsi, Oscar Ortega-Moreno, and Szilárd Révész for valuable conversations on polarization problems.

\medskip

During the preparation of this manuscript, the author used ChatGPT (GPT-5.6 Thinking) as a research assistant.

\bibliographystyle{siam}
\bibliography{InverseSpectral_ref}

\bigskip

\noindent
{\sc Gergely Ambrus}
\smallskip

\noindent
{\em Bolyai Institute, University of Szeged, Hungary \\ and\\ 
 Alfréd Rényi Institute of Mathematics, Budapest, Hungary
 }
\smallskip

\noindent
e-mail address: \texttt{ambrus@server.math.u-szeged.hu}

\end{document}